\documentclass[11pt]{article}

\usepackage{latexsym,amsfonts,amsmath,amsthm,amssymb}

\usepackage[vcentering,dvips]{geometry}
\newtheorem{thm}{Theorem}

\theoremstyle{remark}
\newtheorem{rem}{Remark}

\theoremstyle{definition}

\newcommand{\norm}[2]{\left\|\left.{#1}\right|{#2}\right\|}

\newcommand{\R}{\mathbb{R}}
\newcommand{\Rn}{{\mathbb{R}^n}}
\newcommand{\N}{\mathbb{N}}

\newcommand{\ellqp}{{\ell_{q(\cdot)}(L_{p(\cdot)})}}
\newcommand{\Lpq}{L_{p(\cdot)}(\ell_{q(\cdot)})}
\newcommand{\p}{{p(\cdot)}}
\newcommand{\q}{{q(\cdot)}}

\newcommand{\Lpp}{L_\p(\Rn)}
\renewcommand{\P}{\mathcal{P}(\Rn)}
\newcommand{\esssup}{\operatornamewithlimits{ess-sup}}
\newcommand{\essinf}{\operatornamewithlimits{ess-inf}}

\title{A note on the spaces of variable integrability and summability of Almeida and H\"ast\"o}

\author{Henning Kempka\footnote{\sc Mathematical Institute, Friedrich-Schiller-University, D--07737 Jena, Germany;
\emph{E-mail address:} {\tt henning.kempka@uni-jena.de}},
Jan Vyb\'\i ral \footnote{\sc RICAM, Austrian Academy of Sciences, Altenbergstrasse 69, A--4040 Linz, Austria;
\emph{E-mail address:} {\tt jan.vybiral@oeaw.ac.at}}}

\begin{document}

\maketitle

\begin{abstract}
We address an open problem posed recently by Almeida and H\"ast\"o in \cite{AlHa10}.
They defined the spaces $\ellqp$ of variable integrability and summability and showed
that $\|\cdot|\ellqp\|$ is a norm if $q\ge 1$ is constant almost everywhere or if $1/p(x)+1/q(x)\le 1$
for almost every $x\in \R^n$.
Nevertheless, the natural conjecture (expressed also in \cite{AlHa10}) is that the expression
is a norm if $p(x),q(x)\ge 1$ almost everywhere.
We show that $\|\cdot|\ellqp\|$ is a norm, if $1\le q(x)\le p(x)$ for almost every $x\in\R^n.$ 
Furthermore, we construct an example of $p(x)$ and $q(x)$ with $\min(p(x),q(x))\ge 1$ for every $x\in\R^n$ such that
the triangle inequality does not hold for $\|\cdot|\ellqp\|$.
\end{abstract}

{\bf Subj. Class.}: {Primary 46E35}


\footnotetext{The first author acknowledges the financial support provided by the
DFG project HA 2794/5-1 "Wavelets and function spaces on domains".}
\footnotetext{The second author acknowledges the financial support provided by the
FWF project Y 432-N15 START-Preis "Sparse Approximation and
Optimization in High Dimensions".}
\footnotetext{We would like to thank the referee for useful hints, which helped to improve the paper.}

\section{Introduction}
For the definition of the spaces $\ellqp$ we follow closely \cite{AlHa10}. Spaces of variable integrability $L_\p$ and variable sequence 
spaces $\ell_\q$ have first been considered in 1931 by Orlicz \cite{Orlicz}
but the modern development started with the paper \cite{KovacikRakosnik}.
We refer to \cite{DHHR} for an excellent overview of the vastly growing literature on the subject.

First of all we recall the definition of the 
variable Lebesgue spaces $L_{\p}(\Omega)$, where $\Omega$ is a measurable subset of $\R^n$.
A measurable function $p:\Omega\to(0,\infty]$ is called a variable exponent function if 
it is bounded away from zero. For a set $A\subset\Omega$ we denote $p_A^+=\esssup_{x\in A}p(x)$ and $p_A^-=\essinf_{x\in A}p(x)$; we 
use the abbreviations $p^+=p_\Omega^+$ and $p^-=p_\Omega^-$. The variable exponent Lebesgue space $L_{p(\cdot)}(\Omega)$ consists of 
all measurable functions $f$ such that there exist an $\lambda>0$ such that the modular
\begin{align*}
\varrho_{L_{p(\cdot)}(\Omega)}(f/\lambda)=\int_{\Omega}\varphi_{p(x)}\left(\frac{|f(x)|}{\lambda}\right)dx
\end{align*}
is finite, where
$$
\varphi_p(t)=\begin{cases}t^p&\text{if}\ p\in(0,\infty),\\
0&\text{if}\ p=\infty\ \text{and}\ t\le 1,\\
\infty&\text{if}\ p=\infty\ \text{and}\ t>1.
\end{cases}
$$
This definition is nowadays standard and was used also in \cite[Section 2.2]{AlHa10} and \cite[Definition 3.2.1]{DHHR}.

If we define $\Omega_\infty=\{x\in\Omega:p(x)=\infty\}$ and $\Omega_0=\Omega\setminus\Omega_\infty$, then the Luxemburg norm of a function $f\in L_{p(\cdot)}(\Omega)$ is given by
\begin{align*}
    \norm{f}{L_{p(\cdot)}(\Omega)}&=\inf\{\lambda>0:\varrho_{L_{p(\cdot)}(\Omega)}(f/\lambda)\leq1\}\\
    &=\inf\left\{\lambda>0:\int_{\Omega_0}\!\!\!\left(\frac{|f(x)|}{\lambda}\right)^{p(x)}\!\!\!\!\!\!dx\leq1\ \text{and}\ |f(x)|\le \lambda \ \text{for a.e.}\ x\in\Omega_\infty\right\}.
\end{align*}
If $p(\cdot)\geq1$, then it is a norm, but it is always a quasi-norm if at least $p^->0$, see \cite{KovacikRakosnik} for details.
We denote the class of all measurable functions $p:\Rn\to(0,\infty]$ such that $p^->0$ by $\P$ and the corresponding modular is denoted by $\varrho_\p$ instead of $\varrho_{L_{p(\cdot)}(\R^n)}$.

To define the mixed spaces $\ellqp$ we have to define another modular. For $p,q\in\P$ and a sequence $(f_\nu)_{\nu\in\N_0}$ of 
$\Lpp$ functions we define
\begin{align*}
\varrho_\ellqp(f_\nu)=\sum_{\nu=0}^\infty \inf\left\{\lambda_\nu>0:\varrho_\p\left(\frac{f_\nu}{\lambda_\nu^{1/q(\cdot)}}\right)\leq1\right\},
\end{align*}
where we put $\lambda^{1/\infty}:=1$. The (quasi-) norm in the $\ellqp$ spaces is defined as usually by 
\begin{align*}
\norm{f_\nu}{\ellqp}=\inf\{\mu>0:\varrho_\ellqp(f_\nu/\mu)\leq1\}.
\end{align*}

This (quasi-) norm was used in \cite{AlHa10} to define the spaces of Besov type with variable integrability and summability. 
Spaces of Triebel-Lizorkin type with variable indices have been considered 
recently in \cite{DHR}. The appropriate $\Lpq$ space is a normed space whenever $\essinf_{x\in\R^n}\min(p(x),q(x))\geq1$.
This was the expected result and coincides with the case of constant exponents.

As pointed out in the remark after Theorem 3.8 in \cite{AlHa10}, the same question is still open
for the $\ellqp$ spaces.

\section{When does $\norm{\cdot}{\ellqp}$ define a norm?}
In Theorem 3.6 of \cite{AlHa10} the authors proved that if the condition $\frac{1}{p(x)}+\frac{1}{q(x)}\leq 1$
holds for almost every $x\in\R^n$, then $\norm{\cdot}{\ellqp}$ defines a norm.
They also proved in Theorem 3.8 that $\norm{\cdot}{\ellqp}$ is a quasi-norm for all $p,q\in\P$. 
Furthermore, the authors of \cite{AlHa10} posed a question if the (rather natural) 
condition $p(x),q(x)\geq 1$ for almost every $x\in\R^n$ ensures that $\norm{\cdot}{\ellqp}$ is a norm.

We give (in Theorem \ref{thm1}) a positive answer if $1\leq q(x)\leq p(x)\leq\infty$ almost everywhere on $\R^n.$
Furthermore in Theorem \ref{thm2}, we construct two functions $\p,\q\in\P$ such that $\inf_{x\in\R^n}\min(p(x),q(x))\ge 1$,
but the triangle inequality does not hold for $\norm{\cdot}{\ellqp}$. 

\subsection{Positive results}

We summarize in the following theorem all the cases when the expression $\norm{\cdot}{\ellqp}$ is known to be a norm.
We include the proof of the case discussed already in \cite{AlHa10} for the sake of completeness.

\begin{thm}\label{thm1}
Let $p,q\in\P$ such that either $p(x)\ge 1$ and $q \ge 1$ is constant almost everywhere,
or $1\leq q(x)\leq p(x)\leq\infty$ for almost every $x\in\R^n$, or
$1/p(x)+1/q(x)\le 1$ for almost every $x\in\R^n$. Then $\norm{\cdot}{\ellqp}$ defines a norm.
\end{thm}
\begin{proof}
If $p(x)\ge 1$ and $q \ge 1$ is constant almost everywhere, then the proof is trivial.

In the remaining cases, we want to show that
\begin{equation*}
\|f_\nu+g_\nu|\ellqp\|\le \|f_\nu|\ellqp\|+\|g_\nu|\ellqp\|
\end{equation*}
for all sequences of measurable functions $\{f_\nu\}_{\nu\in\N_0}$ and $\{g_\nu\}_{\nu\in\N_0}$.
Let $\mu_1>0$ and $\mu_2>0$ be given with
$$
\varrho_\ellqp\left(\frac{f_\nu}{\mu_1}\right)\le 1\quad \text{and}\quad
\varrho_\ellqp\left(\frac{g_\nu}{\mu_2}\right)\le 1.
$$
We want to show that
$$
\varrho_\ellqp\left(\frac{f_\nu+g_\nu}{\mu_1+\mu_2}\right)\le 1.
$$
For every $\varepsilon>0$, there exist sequences of positive numbers $\{\lambda_\nu\}_{\nu\in\N_0}$
and $\{\Lambda_\nu\}_{\nu\in\N_0}$ such that
\begin{equation}\label{eq:have1}
\varrho_\p\left(\frac{f_\nu(x)}{\mu_1\lambda_{\nu}^{1/q(x)}}\right)\le 1\quad \text{and}\quad
\varrho_\p\left(\frac{g_\nu(x)}{\mu_2\Lambda_{\nu}^{1/q(x)}}\right)\le 1
\end{equation}
together with
$$
\sum_{\nu=0}^\infty\lambda_\nu\le 1+\varepsilon\quad\text{and}\quad \sum_{\nu=0}^\infty\Lambda_\nu\le 1+\varepsilon.
$$
We set
$$
A_\nu:=\frac{\mu_1\lambda_\nu+\mu_2\Lambda_\nu}{\mu_1+\mu_2}, \quad \text{i.e.}\quad \sum_{\nu=0}^\infty A_\nu\le 1+\varepsilon.
$$
We shall prove that
\begin{equation}\label{eq:todo}
\varrho_{\p}\left(\frac{f_\nu(x)+g_\nu(x)}{A_\nu^{1/q(x)}(\mu_1+\mu_2)}\right)\le 1\quad\text{for all}\quad \nu\in\N_0.
\end{equation}
Let $\Omega_0:=\{x\in\R^n: p(x)<\infty\}$ and $\Omega_\infty:=\{x\in\R^n:p(x)=\infty\}$. We put for every $x\in\Omega_0$
$$
F_\nu(x):=\left(\frac{|f_\nu(x)|}{\mu_1\lambda_\nu^{1/q(x)}}\right)^{p(x)}\quad\text{and}\quad
G_\nu(x):=\left(\frac{|g_\nu(x)|}{\mu_2\Lambda_\nu^{1/q(x)}}\right)^{p(x)}.
$$
Then \eqref{eq:have1} may be reformulated as
\begin{equation}\label{eq:have11}
\int_{\Omega_0}F_\nu(x)dx\le 1\quad \text{and}\quad \esssup_{x\in\Omega_\infty}\frac{|f_\nu(x)|}{\mu_1\lambda_\nu^{1/q(x)}}\le 1
\end{equation}
and
\begin{equation}\label{eq:have12}
\int_{\Omega_0}G_\nu(x)dx\le 1\quad\text{and}\quad \esssup_{x\in\Omega_\infty}\frac{|g_\nu(x)|}{\mu_2\Lambda_\nu^{1/q(x)}}\le 1\ .
\end{equation}
Our aim is to prove \eqref{eq:todo}, which reads
\begin{equation}\label{eq:todo2}
\int_{\Omega_0}\left(\frac{|f_\nu(x)+g_\nu(x)|}{A_\nu^{1/q(x)}(\mu_1+\mu_2)}\right)^{p(x)}dx\le 1\quad\text{and}\quad \esssup_{x\in\Omega_\infty}
\frac{|f_\nu(x)+g_\nu(x)|}{A_\nu^{1/q(x)}(\mu_1+\mu_2)}\le 1.
\end{equation}
We first prove the second part of \eqref{eq:todo2}. First we observe that \eqref{eq:have11} and \eqref{eq:have12}
imply
$$
|f_\nu(x)|\le \mu_1\lambda_\nu^{1/q(x)}\quad \text{and}\quad |g_\nu(x)|\le \mu_2\Lambda_\nu^{1/q(x)}
$$
holds for almost every $x\in\Omega_\infty.$
Using $q(x)\ge 1$, and H\"older's inequality in the form
$$
\frac{\mu_1\lambda_\nu^{1/q(x)}+\mu_2\Lambda_\nu^{1/q(x)}}{\mu_1+\mu_2}\le \left(\frac{\mu_1\lambda_\nu+\mu_2\Lambda_\nu}{\mu_1+\mu_2}\right)^{1/q(x)},
$$
we get
$$
\frac{|f_\nu(x)+g_\nu(x)|}{A_\nu^{1/q(x)}(\mu_1+\mu_2)} \le 1.
$$
If $q(x)=\infty$, only notational changes are necessary.

Next we prove the first part of \eqref{eq:todo2}. Let $1\le q(x)\le p(x)<\infty$ for almost all $x\in\Omega_0$. Then we use H\"older's inequality in the form
\begin{align}\label{eq:Hol1}
F&_\nu(x)^{1/p(x)}\lambda_\nu^{1/q(x)}\mu_1+G_\nu(x)^{1/p(x)}\Lambda_\nu^{1/q(x)}\mu_2\\
&\notag \le(\mu_1+\mu_2)^{1-1/q(x)}
(\mu_1\lambda_\nu+\mu_2\Lambda_\nu)^{1/q(x)-1/p(x)}
(F_\nu(x)\lambda_\nu\mu_1+G_\nu(x)\Lambda_\nu\mu_2)^{1/p(x)}.
\end{align}
If $1/p(x)+1/q(x)\le 1$ for almost every $x\in\Omega_0$, then we replace \eqref{eq:Hol1} by
\begin{align}\label{eq:Hol2}
F&_\nu(x)^{1/p(x)}\lambda_\nu^{1/q(x)}\mu_1+G_\nu(x)^{1/p(x)}\Lambda_\nu^{1/q(x)}\mu_2\\
\notag&\le(\mu_1+\mu_2)^{1-1/p(x)-1/q(x)}
(\mu_1\lambda_\nu+\mu_2\Lambda_\nu)^{1/q(x)}
(F_\nu(x)\mu_1+G_\nu(x)\mu_2)^{1/p(x)}.
\end{align}
Using \eqref{eq:Hol1}, we may further continue
\begin{align*}
\int_{\Omega_0}&\left(\frac{|f_\nu(x)+g_\nu(x)|}{A_\nu^{1/q(x)}(\mu_1+\mu_2)}\right)^{p(x)}dx\\
&\hspace{-2em}= \int_{\Omega_0} \left(\frac{F_\nu(x)^{1/p(x)}\lambda_\nu^{1/q(x)}\mu_1+G_\nu(x)^{1/p(x)}\Lambda_\nu^{1/q(x)}\mu_2}{\mu_1+\mu_2}\right)^{p(x)}
 \!\!\!\!\cdot \left(\frac{\mu_1\lambda_\nu+\mu_2\Lambda_\nu}{\mu_1+\mu_2}\right)^{-\frac{p(x)}{q(x)}}\!\!\!dx\\
&\hspace{-2em}\le \int_{\Omega_0} \frac{F_\nu(x)\lambda_\nu\mu_1+G_\nu(x)\Lambda_\nu\mu_2}{\mu_1\lambda_\nu+\mu_2\Lambda_\nu}dx\\
&\hspace{-2em}= \frac{\mu_1\lambda_\nu}{\mu_1\lambda_\nu+\mu_2\Lambda_\nu}\int_{\Omega_0} F_\nu(x)dx
+ \frac{\mu_2\Lambda_\nu}{\mu_1\lambda_\nu+\mu_2\Lambda_\nu}\int_{\Omega_0} G_\nu(x)dx\le 1,
\end{align*}
where we used also \eqref{eq:have11} and \eqref{eq:have12}.
If we start with \eqref{eq:Hol2} instead, we proceed in the following way
\begin{align*}
\int_{\Omega_0}&\left(\frac{|f_\nu(x)+g_\nu(x)|}{A_\nu^{1/q(x)}(\mu_1+\mu_2)}\right)^{p(x)}dx\\
&\hspace{-2em}= \int_{\Omega_0} \left(\frac{F_\nu(x)^{1/p(x)}\lambda_\nu^{1/q(x)}\mu_1+G_\nu(x)^{1/p(x)}\Lambda_\nu^{1/q(x)}\mu_2}{\mu_1+\mu_2}\right)^{p(x)}
 \!\!\!\!\cdot \left(\frac{\mu_1\lambda_\nu+\mu_2\Lambda_\nu}{\mu_1+\mu_2}\right)^{-\frac{p(x)}{q(x)}}\!\!\!dx\\
&\hspace{-2em}\le \int_{\Omega_0} \frac{F_\nu(x)\mu_1+G_\nu(x)\mu_2}{\mu_1+\mu_2}dx= \frac{\mu_1}{\mu_1+\mu_2}\int_{\Omega_0} F_\nu(x)dx
+ \frac{\mu_2}{\mu_1+\mu_2}\int_{\Omega_0} G_\nu(x)dx\le 1.
\end{align*}
In both cases, this finishes the proof of \eqref{eq:todo2}.
\end{proof}
\begin{rem}
\begin{enumerate}
\item[(i)] A simpler proof of Theorem \ref{thm1} is possible (and was proposed to us by the referee) if $1\le q(x)\le p(x) \le \infty$.
Namely, if $1\le q\le p\le \infty, \lambda>0$ and $t\ge 0$, then
\begin{equation}\label{eq:eqq1}
\varphi_p\left(\frac{t}{\lambda^{1/q}}\right)=\varphi_{\frac{p}{q}}\left(\frac{\varphi_q(t)}{\lambda}\right),
\end{equation}
where we use the convention that $\frac{p}{q}=1$ if $p=q=\infty.$ This allows to simplify the modular $\varrho_{\ellqp}$ to
\begin{equation}\label{eq:eqq2}
\varrho_{\ellqp}(f_\nu)=\sum_{\nu=0}^\infty \left\|\varphi_{q(\cdot)}(|f_\nu|)\right\|_{\frac{p(\cdot)}{q(\cdot)}}.
\end{equation}
This shows that $\varrho_{\ellqp}(f_\nu)$ is a composition of only convex functions. Hence, it is a convex modular
and therefore it induces a norm. Unfortunately, we were not able to find such a simplification for the case $1/p(x)+1/q(x)\le 1$.
The advantage of our proof of Theorem \ref{thm1} is that it proves both the cases in a unified way.
\item[(ii)] Let us observe that \eqref{eq:eqq1} loses its sense if $p<q=\infty$. This shows, why \eqref{eq:eqq2}
(which was already used in \cite{AlHa10} for $q^+<\infty$) has to be applied with certain care.
\item[(iii)] The method of the proof of Theorem \ref{thm1} can be actually used to show that under the conditions
posed on $\p$ and $\q$ in Theorem \ref{thm1},  $\varrho_\ellqp$ is a convex modular,
which is a stronger result than the norm property. 
\end{enumerate}
\end{rem}

\subsection{Counterexample}
\begin{thm}\label{thm2}
There exist functions $p,q\in\P$ with $\inf_{x\in\R^n}p(x)\ge 1$ and $\inf_{x\in\R^n}q(x)\ge 1$ such that
$\|\cdot|\ellqp\|$ does not satisfy the triangle inequality.
\end{thm}
\begin{proof}
Let $Q_0,Q_1\subset \R^n$ be two disjoint unit cubes, let $p(x):=1$ everywhere on $\R^n$
and put $q(x):=\infty$ for $x\in Q_1$ and $q(x):=1$ for $x\not\in Q_1$.
Let $f_1=\chi_{Q_0}$ and $f_2=\chi_{Q_1}$. Finally, we put $f=(f_1,f_2,0,\dots)$ and $g=(f_2,f_1,0,\dots)$.

We calculate for every $L>0$ fixed
$$
\inf\left\{\lambda_1>0: \varrho_{\p}\left(\frac{f_1(x)}{\lambda_1^{1/q(x)}L}\right)\le 1\right\}
=\inf\left\{\lambda_1>0: \frac{1}{\lambda_1 L}\le 1\right\}=1/L
$$
and
$$
\inf\left\{\lambda_2>0: \varrho_{\p}\left(\frac{f_2(x)}{\lambda_2^{1/q(x)}L}\right)\le 1\right\}
=\inf\left\{\lambda_2>0: \frac{1}{L}\le 1\right\}.
$$
If $L\ge 1$, then the last expression is equal to zero, otherwise it is equal to $\infty$.

We obtain
\begin{align*}
\|f|\ellqp\|=\inf\{L>0:\varrho_\ellqp(f/L)\le 1\}= \inf\{L>0:1/L+0\le 1\}=1
\end{align*}
and the same is true also for $\|g|\ellqp\|$. It is therefore enough to show that $\|f+g|\ellqp\|>2$.

Using the calculation
\begin{align*}
\inf&\left\{\lambda>0:\varrho_{\p}\left(\frac{f_1(x)+f_2(x)}{L\cdot\lambda^{1/q(x)}}\right)\le 1\right\}
=\inf\left\{\lambda>0:\int_{Q_0}\frac{1}{L\cdot\lambda}+\int_{Q_1}\frac{1}{L}\le 1\right\}\\
&\qquad =\inf\left\{\lambda>0:\frac{1}{L\cdot\lambda}+\frac{1}{L}\le 1\right\}=\frac{1}{L-1},
\end{align*}
which holds for every $L>1$ fixed, we get
\begin{align*}
\|f+g|\ellqp\|&=\inf\left\{L>0:\varrho_\ellqp\left(\frac{f+g}{L}\right)\le 1\right\}\\
&=\inf\left\{L>0:2\inf\left\{\lambda>0:\varrho_{\p}\left(\frac{f_1(x)+f_2(x)}{L\cdot\lambda^{1/q(x)}}\right)\le 1\right\}\le 1\right\}\\
&=\inf\left\{L>1:2\cdot\frac{1}{L-1}\le 1\right\}=3.
\end{align*}
\end{proof}

\begin{rem}
Let us observe that $1\le q(x)\le p(x)\le \infty$ holds for $x\in Q_0$ and $1/p(x)+1/q(x)\le 1$ is true for $x\in Q_1$.
It is therefore necessary to interpret the assumptions of Theorem \ref{thm1} in a correct way,
namely that one of the conditions of Theorem \ref{thm1} holds for (almost) all $x\in\R^n$. This is not
to be confused with the statement that for (almost) every $x\in\R^n$ at least one of the conditions
is satisfied, which is not sufficient.
\end{rem}

\begin{rem}
A similar calculation (which we shall not repeat in detail) shows that one may also put $q(x):=q_0$ large enough for $x\in Q_1$
to obtain a counterexample. Hence there is nothing special about the infinite value of $q$ and the same counterexample
may be reproduced with uniformly bounded exponents $p,q\in\P$.
\end{rem}

\bibliographystyle{amsplain}

\end{document}